\documentclass[12pt]{amsart}

\usepackage{graphicx}

\usepackage{amsfonts}
\usepackage{amssymb}
\usepackage{amsmath}
\usepackage{latexsym}
\usepackage{amscd}
\usepackage{lscape}
\usepackage{pifont}
\usepackage{fancyhdr}
\usepackage{amsthm}
\usepackage{amsrefs}

\newtheorem{thm}{Theorem}[section]

\newtheorem{lem}[thm]{Lemma}
\newtheorem{prop}[thm]{Proposition}
{\bf}{\it}
{\bf}{\it}
\newtheorem*{riemann-mapping-theorem}{Riemann Mapping  Theorem}{\bf}{\it}
{\bf}{\it}
{\bf}{\it}
{\bf}{\it}
{\bf}{\it}
{\bf}{\it}
{\bf}{\it}
{\bf}{\it}
{\bf}{\it}
{\bf}{\it}
\newtheorem*{thurstonthm}{Thurston Theorem}{\bf}{\it}

\newenvironment{pf*}[1]{\proof[#1]}{\endproof}
\usepackage{euscript}

\newcommand{\cal}[1]{{\mathcal #1}}

\newcommand{\beq}{\begin{equation}}
\newcommand{\eeq}{\end{equation}}

\newtheorem{defn}{Definition}[section]

\theoremstyle{remark}
\newtheorem{rem}{Remark}[section]

\renewcommand{\mod}{\operatorname{mod}}

\renewcommand{\Im}{\operatorname{Im}}

\numberwithin{equation}{section}
\newcommand{\thmref}[1]{Theorem~\ref{#1}}
\newcommand{\propref}[1]{Proposition~\ref{#1}}

\newcommand{\lemref}[1]{Lemma~\ref{#1}}

\newcommand{\cM}{{\cal M}}

\newcommand{\cT}{{\cal T}}
\newcommand{\cN}{{\cal N}}

\newcommand{\cI}{{\cal I}}

\newcommand{\cS}{{\cal S}}

\newcommand{\CC}{{\mathbb C}}
\newcommand{\CCC}{{\scriptstyle\CC}}
\newcommand{\RR}{{\mathbb R}}

\newcommand{\ZZ}{{\mathbb Z}}
\newcommand{\NN}{{\mathbb N}}

\newcommand{\ignore}[1]{{}}

\begin{document}

\title[Thurston equivalence is decidable]{Thurston equivalence to a rational map is decidable}
\author{Sylvain Bonnot, Mark Braverman, Michael Yampolsky}
\date{May 9, 2010}
\maketitle
\begin{abstract}
We demonstrate that the question whether or not a given topological ramified covering map of the 2-sphere is Thurston equivalent
to a rational map is algorithmically decidable. 
\end{abstract}

\maketitle

\section{Introduction}
\label{section0}

This paper solves a long-standing open problem in one-dimensional Complex Dynamics: we show that Thurston's equivalence to a postcritically finite rational 
map is algorithmically decidable. Thurston's theorem \cite{DH} is central to the subject. In the case when a rational mapping exists, it is essentially 
unique, and the proof of the theorem \cite{DH} supplies an iterative algorithm for computing its coefficients. When, for instance, the rational mapping is 
a quadratic polynomial, the Spider algorithm of Hubbard and Schleicher \cite{HS}, computes the coefficients starting from a 
convenient combinatorial description of the branched covering. However, the algorithm will go astray if the branched covering data cannot be realized by 
a polynomial. Thus the question we answer in this work is both natural and important.

Here is the Equivalence problem we consider in this note:

\medskip
\noindent
{\bf Problem: Equivalence to a Rational Map.}
 {\sl Given a piecewise linear post-critically finite ramified covering map $f$ from $\hat \CC$ to itself, determine whether or not 
it is Thurston equivalent to a rational map. If it is equivalent, give an algorithm to compute the coefficients of the corresponding rational map
 (defined up to a conjugacy with a   M{\"o}bius map).}

\medskip

Our main result is:
\begin{thm}
\label{main thm}
The problem of equivalence to a rational map is algorithmically solvable.
\end{thm}

\section{Thurston mappings}
\label{section1}
In this section we recall the basic setting of Thurston's characterization of rational functions.
\subsection{Ramified covering maps}
\label{section1.1}
Let $f:S^{2}\rightarrow S^{2}$ be an orientation-preserving branched covering map of the two-sphere. We define the \textit{postcritical set $P_{f}$} by
\[
 P_{f}:=\bigcup_{n>0}f^{\circ n}(\Omega_{f}),
\]
where $\Omega_{f}$ is the set of critical points of $f$. When the postcritical set $P_{f}$ is finite we say that $f$ is a \textit{Thurston mapping}.
\paragraph*{\bf Thurston equivalence.} 
Two Thurston maps $f$ and $g$ are \textit{Thurston equivalent} if there are homeomorphisms $\phi_{0}, \phi_{1}:S^{2}\rightarrow S^{2}$ such that
\begin{enumerate}
 \item the maps $\phi_{0}, \phi_{1}$ coincide on $P_{f}$, send $P_{f}$ to $P_{g}$ and are isotopic \text{rel } $P_{f}$; 
\item the diagram
\[
\begin{CD}
S^{2} @>\phi_{1}>> S^{2}\\
@VVfV @VVgV\\
S^{2} @>\phi_{0}>> S^{2}
\end{CD}
\]
commutes.
\end{enumerate}
 
\paragraph{\bf Orbifold of a Thurston map.} Given a Thurston map $f:S^{2} \rightarrow S^{2}$, we define a function $N_{f}:S^{2} \rightarrow \mathbb{N}\cup {\infty}$ as follows:
\[
N_{f}(x)=\begin{cases}
1& \text{if $x \notin P_{f}$},\\
\infty & \text{if $x$ is in a cycle containing a critical point},\\
\underset{f^{k}(y)=x}{\text{lcm}} \text{deg}_{y}(f^{\circ k}) & \text{ otherwise}.
\end{cases}
\]

The pair $(S^{2},N_{f})$ is called the \textit{orbifold of $f$}. 
The {\it signature} of the orbifold $(S^2,N_f)$ is the set $\{N_f(x)\text{ for }x\text{ such that }1<N_f(x)<\infty\}$. 
The \textit{Euler characteristic} of the orbifold is given by 
\[
 \chi(S^{2},N_{f}):= 2-\sum_{x \in P_{f}}\left(1-\frac{1}{N_{f}(x)}\right).
\]
One can prove that $\chi(S^{2},N_{f})\leq 0$. In the case where $\chi(S^{2},N_{f})< 0$, we say that the orbifold is \textit{hyperbolic}. Observe that most orbifolds are hyperbolic: indeed, as soon as the cardinality $\vert P_{f} \vert >4$, the orbifold is hyperbolic.

\medskip
\paragraph{\bf Thurston linear transformation.} We recall that a simple closed curve $\gamma \subset S^{2}-P_{f}$ is \textit{essential} if it does not bound a disk, is \textit{non-peripheral} if it does not bound a punctured disk.

\begin{defn}A  \textit{multicurve} $\Gamma$ on $(S^{2},P_{f})$ is a set of disjoint, nonhomotopic, essential, nonperipheral simple closed curves on $S^{2}-P_{f}$.
A multicurve $\Gamma$ is \textit{f-stable} if for every curve $\gamma \in \Gamma$, each component $\alpha$ of $f^{-1}(\gamma)$ is either trivial (meaning inessential or peripheral) or homotopic rel $P_{f}$ to an element of $\Gamma$. 

\end{defn}

To any $f$-stable multicurve is associated its \textit{Thurston linear transformation} $f_{\Gamma}:\mathbb{R}^{\Gamma}\rightarrow \mathbb{R}^{\Gamma}$, best described by the following transition matrix
\[
 M_{\gamma \delta}=\sum_{\alpha} \frac{1}{\text{deg}(f:\alpha \rightarrow \delta)}
\]
where the sum is taken over all the components $\alpha$ of $f^{-1}(\delta)$ which are isotopic rel $P_{f}$ to $\gamma$.
Since this matrix has nonnegative entries, it has a leading eigenvalue $\lambda(\Gamma)$ that is real and nonnegative (by the Perron-Frobenius theorem).
\begin{figure}
\centerline{\includegraphics[width=\textwidth]{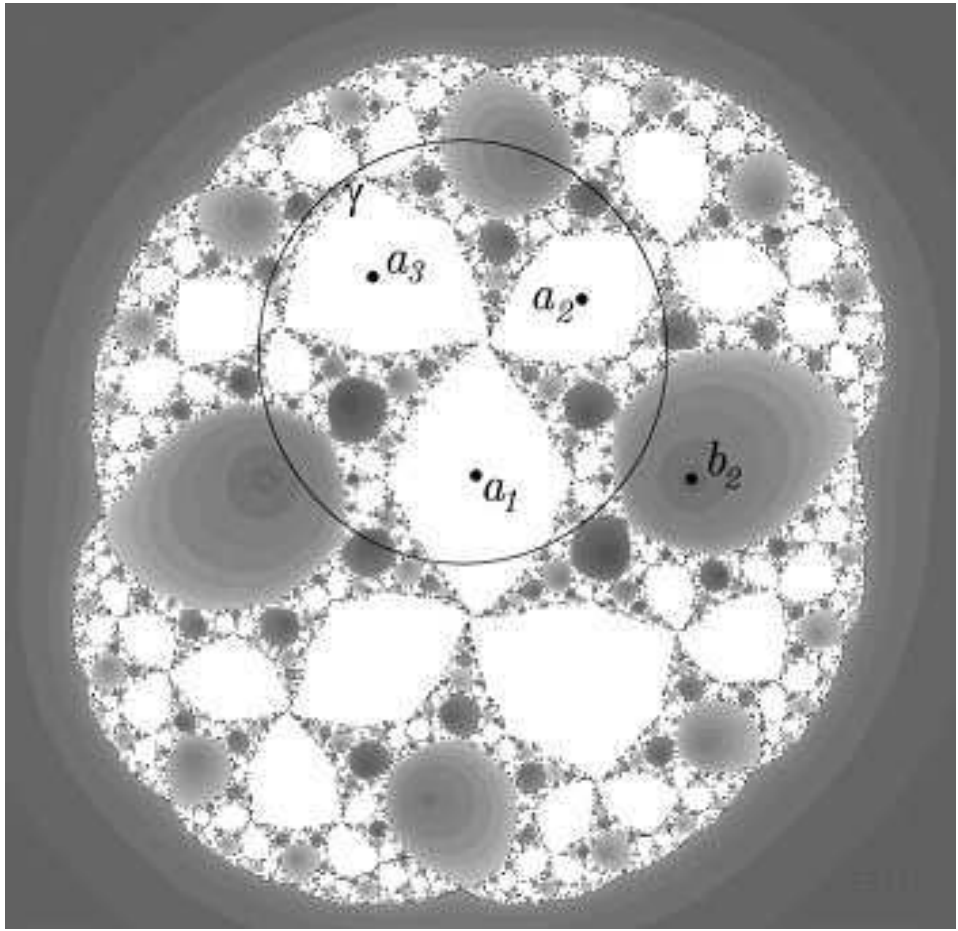}}
\caption{A stable multicurve of a quadratic rational map.\label{fig-example-gamma}}
\end{figure}

We can now state Thurston's theorem:

\begin{thurstonthm} Let $f:S^{2} \rightarrow S^{2}$ be a Thurston map
with hyperbolic orbifold. Then $f$ is Thurston equivalent to a rational function $g$ if and only if $\lambda(\Gamma)<1$ for every $f$-stable multicurve $\Gamma$. The rational function $g$ is unique up to conjugation with an automorphism of $\mathbb{P}^{1}$. 
\end{thurstonthm}

When a stable multicurve $\Gamma$ has a leading eigenvalue $\lambda(\Gamma)\geq 1$, we call it a \textit{Thurston obstruction}.

\subsection*{Several examples}
Let us first give an example of a quadratic rational map $f$ with an $f$-stable multicurve $\Gamma$.
The map is given by the formula:
$$f(z)=\frac{z^2+c}{z^2-1},\text{ with }c=\frac{1+i\sqrt{3}}{2}.$$
The picture of its Julia set is seen in Figure \ref{fig-example-gamma}, it is popularized as the cover art of the Stony Brook preprint series.
The map $f$ is known as the {\it mating} of two quadratic Julia sets: Douady's rabbit and the basilica (see e.g. \cite{Pil}).

The two critical points of $f$ are $a_1=0$ and $b_1=\infty$. Both of them are periodic:
$$a_1=0\overset{f}{\mapsto} a_2\overset{f}{\mapsto} a_3\overset{f}{\mapsto} a_1,$$
$$b_1=\infty\overset{f}{\mapsto} b_2=1\overset{f}{\mapsto} b_1.$$
Our stable multicurve $\Gamma$ consists of a single simple closed curve $\gamma$ which separates $a_i$'s from $b_i$'s.
It is easy to see that the corresponding transition matrix consists of a single entry $1/2$. Thus,
$\lambda(\Gamma)=1/2$.

To give an example of a Thurston obstruction, we will need to work a little harder. We again use the procedure known as {\it mating}.
Let us again start with Douady's rabbit polynomial, $f_c(z)=z^2+c$ 
which is the unique quadratic polynomial with $\Im c>0$ such that the critical point $0$ is peridic with period $3$. Thus, the 
postcritical set of $f_c$ is $\{w_0=0,w_1=c,w_2=c^2+c,w_3=\infty\}$. Consider also the complex conjugate,
the polynomial $f_{\bar c}$ whose postcritical set we denote 
$\{w'_0=0,w'_1=\bar c,w'_2=\bar c^2+\bar c,w'_3=\infty\}$. The {\it formal mating} of these two polynomials is the branched covering mapping of $S^2$ which is obtained as follows.
We first compactify the complex plane by adjoining a circle of directions at infinity, $\{\infty\cdot e^{2\pi i\theta}|\;\theta\in \RR/\ZZ\}$. We denote such 
compactification with the natural topology by $\textcircled{$\CCC$}$. Let us now glue two copies $\textcircled{$\CCC$}_1$, $\textcircled{$\CCC$}_2$  along the circles at infinity
using the equivalence relation $\sim_\infty$ given by 
$$\infty\cdot e^{2\pi i\theta}\in \textcircled{$\CCC$}_1 \sim_\infty \infty\cdot e^{-2\pi i\theta}\in \textcircled{$\CCC$}_2.$$
Evidently,
$$S=\textcircled{$\CCC$}_1\sqcup\textcircled{$\CCC$}_2/\sim_\infty\simeq S^2.$$
The formal mating of $f_c$ and $f_{\bar c}$ is the well-defined branched covering map $F$ of the 2-sphere $S$ which is given by
$f_c$ on $\textcircled{$\CCC$}_1$ and $f_{\bar c}$ on $\textcircled{$\CCC$}_2.$ By construction, this map has an invariant equator (the two circles at $\infty$ glued together), and its postcritical set is the union
$$\{w_0,w_1,w_2\}\cup \{w'_0,w'_1,w'_2\}.$$
An obstruction for this mapping is given by a multicurve $\Gamma$ consisting of three loops $\gamma_i$ separating $w_i$, $w'_i$ from the rest of the postcritical set
(see Figure \ref{fig-example-obstruction}). It is easy to see that $\Gamma$ is an $F$-stable multicurve, with the associated transition matrix
$$\left(\begin{array}{ccc}
0&0&1\\
1&0&0\\
0&1&0\\
\end{array}\right)$$
so that $\lambda(\Gamma)=1.$
\begin{figure}
\centerline{\includegraphics[width=\textwidth]{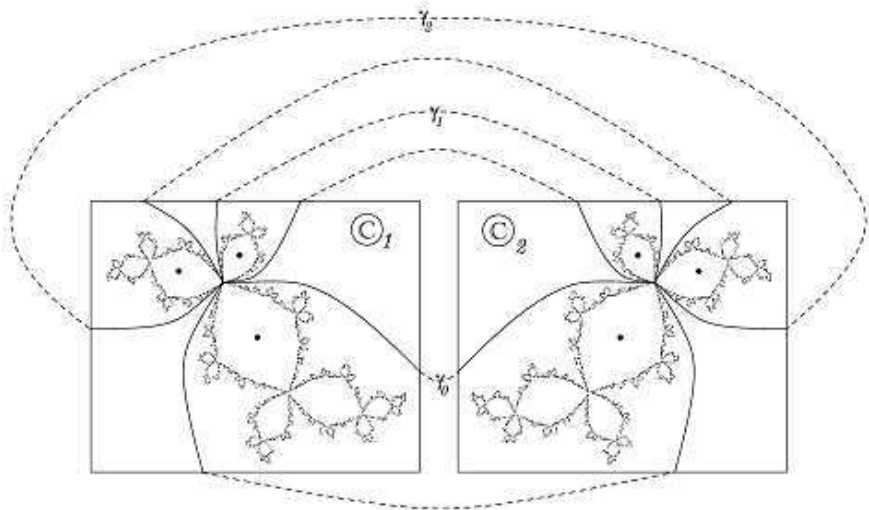}}
\caption{An example of Thurston obstruction.\label{fig-example-obstruction}}
\end{figure}

\subsection{A piecewise-linear Thurston mapping}\label{section2.1}
For the purposes of an algorithmic analysis, we will require a finite description of a branched covering $f:S^2\to S^2$.

Since we will work mainly in the piecewise linear category, it is convenient to recall here some definitions.

\paragraph{\bf Simplicial complexes} Following \cite{thu} (chapter 3.2 and 3.9) we call {\it a simplicial complex} any locally finite collection 
$\Sigma$ of simplices  satisfying the following two conditions: 
\begin{itemize}
\item a face of a simplex in $\Sigma$ is also in $\Sigma$, and 
\item the intersection of any two simplices in $\Sigma$ is either empty or a face of both. 
\end{itemize}
The union of all simplices in $\Sigma$ is the {\it polyhedron} of $\Sigma$ (written $\vert \Sigma \vert$). 

\paragraph{\bf Piecewise linear maps} A map $f:M \rightarrow N$ from a subset of an affine space into another affine space is {\it piecewise linear (PL)} if it is the restriction of a simplicial map defined on the polyhedron of some simplicial complex.

We also define {\it piecewise linear (PL) manifolds} as manifolds having an atlas where the transition maps between overlapping charts are piecewise linear homeomorphisms between open subsets of $\mathbb{R}^{n}$. It is well known that any piecewise linear manifold has a triangulation:  there is a simplicial complex $\Sigma$ together with a homeomorphism $\vert \Sigma \vert \rightarrow X$ which is assumed a PL map (see \cite{thu}, proof of theorem 3.10.2).

One example of such a manifold is the standard piecewise linear (PL) 2-sphere, which is nicely 
described in \cite{thu} as follows: pick any convex 3-dimensional polyhedron
 $K \subset \mathbb{R}^{3}$, and consider the charts corresponding to all the 
possible orthogonal projections of the boundary (topological) sphere $\partial K$ onto hyperplanes in $\mathbb{R}^{3}$. 
The manifold thus obtained is the {\it standard piecewise linear  2-sphere}. One can prove that another choice of 
polyhedron would lead to an isomorphic object (see exercise 3.9.5 in \cite{thu}).

It is known that in dimension three or lower, every topological manifold has a PL structure, and any two such structures are PL equivalent (in dimension 2, see \cite{Rad}, for the dimension 3 consult \cite{Bin}). 

\paragraph{\bf Piecewise linear branched covers.}
We begin by formulating the following proposition which describes how to lift a triangulation by a PL branched cover (see \cite{Doug},section 6.5.4):

\begin{prop}[{\bf Lifting a triangulation}]
\label{lift cover}
Let $B$ be a compact topological surface, $\pi:X\to B$ a finite ramified cover of $B$. Let $\Delta$ be the set of branch points of $\pi$, 
and let $\cT$ be a triangulation of $B$ such that $\Delta$ is a subset of vertices of $\cT$ ($\Delta\subset K_{0}(\cT)$ in the established notation).
Then there exists a triangulation $\cT'$ of $X$, unique up to a bijective change of indices, 
so that the branched covering map $\pi:X \rightarrow B$ sends vertices to vertices, edges to edges and faces to faces. Moreover, if $X=B$ is a standard PL
2-sphere and $\pi$ is PL, then
$\cT'$ can be produced constructively given a description of $\cT$.
\end{prop}

We consider PL maps of the standard PL 2-sphere which are topological branched coverings with a finite number of branch points. 
We call such a map a {\it piecewise linear Thurston mapping}. 

\begin{rem}
\label{representation-PL}
Note that any such covering may
be realized as a piecewise-linear branched covering map of a triangulation of $\CC$ with rational vertices. An algorithmic description of a PL branched covering
could thus either be given by the combinatorial data describing the simplicial map, or  as a collection of affine maps of triangles in $\hat\CC$ 
with rational vertices. We will alternate between these descriptions as convenient. 
\end{rem}

We note:

\begin{prop}
\label{ThPL}
Every Thurston mapping $f$ is Thurston equivalent to a PL Thurston mapping.
\end{prop}

Before proving the above Proposition, let us formulate a basic topological fact, known as Alexander's trick:

\medskip
\noindent
{\bf Alexander's trick.}
{\it Two homeomorphisms of the closed $n$-dimensional ball, which are isotopic on the boundary, are isotopic.}

\medskip
\noindent

\begin{proof}[Proof of \propref{ThPL}]
We may start with a triangulation $\cT_1$ of $S^2$ whose vertices include the postcritical set $P_f$. 
Refining the triangulation $\cT_1$ to $\cT_2$, if necessary, we isotope $f$ to a map which leaves the vertices and the edges of $\cT_2$ invariant.
Finally, every topological 
map from a triangle to a triangle can be isotoped into a simplicial map using Alexander's trick.
We can thus further isotope our map to a PL Thurston mapping with triangulation $\cT_2$.
\end{proof}

\section{Outline of the proof of \thmref{main thm}}

The proof will rely on a construction of two explicit algorithms, $A_1$ and $A_2$, which, given a postcritically finite piecewise linear branched covering  $f:S^2\to S^2$
with a hyperbolic orbifold, perform the following tasks:
\begin{itemize}
\item[$A_1$] If $f$ has a Thurston obstruction, the algorithm $A_1$ will terminate and output the obstruction. It will not terminate otherwise.
\item[$A_2$] If $f$ is Thurston equivalent to a rational mapping $R$, then the algorithm $A_2$ will terminate. It will identify the rational mapping by outputting a ball in an
appropriate parameter space of rational maps which isolates the rational mapping $R$ from postcritically finite mappings of the same degree and with the same size of the postcritical
set. If $f$ is not equivalent to any rational mapping, then the algorithm $A_2$ will not terminate.
\end{itemize}

We further will use a polynomial root-finding algorithm $A_3$ which finds an isolated root $\bar x_*\in\RR^m$ of a system of polynomial 
equations $\{P_i(\bar x)=0\}$.
\begin{itemize}
\item[$A_3$] the input of the algorithm is: a system of polynomial equations $\{P_i(\bar x)=0\}$ for $\bar x\in\CC^m$
 (the coefficients of the polynomials $P_i$ are either given through an oracle, or computed
with an arbitrarily high precision via a given algorithm); a rational ball $B(\bar w, r)\subset \CC^m$ which contains  $\bar x_*$ and such that
$B(\bar w,2r)$ does not contain any other roots; a natural number $n$. The output is $\bar d_n\in\CC^m$ with the property $||\bar x_*-\bar d_n||<2^{-n}$. 
\end{itemize}

\begin{proof}[Proof of \thmref{main thm} assuming the existence of $A_1$ and $A_2$.]
Given a postcritically finite piecewise linear map $f$ with a hyperbolic orbifold,
we will run the two algorithms $A_1$ and $A_2$ in parallel. One and only one of them will terminate. If it is $A_1$, then
we conclude that $f$ is not equivalent to any rational map. If it is $A_2$ then we know that a Thurston equivalent rational map $R$ exists, and we are given an isolating neighborhood
for it in the parameter space. The root-finding algorithm $A_3$ can then be employed to find the coefficients of $R$ with any given precision.
\end{proof}

\section{Some topological preliminaries}
\label{section2.2}

\medskip
\paragraph{\bf Simple closed curves.}

Recall that two curves are in a {\it minimal position} if they realize the minimal number of intersections in their homotopy classes.
\begin{lem}[{\bf The Bigon Criterion}]
Two transverse simple closed curves on a surface S are in a minimal position if and only if the two arcs between any pair of intersection points do not bound an embedded disk in S. 
\end{lem}
Let us also formulate an elementary fact:
\begin{lem}
\label{homotopy check1}
Two simple closed curves on a surface $S$ are homotopic if and only if they can be isotoped to boundary curves of an annulus.
\end{lem}

We now prove:
\begin{prop}
\label{homotopy check}
There exists an algorithm to check whether two simple closed polygonal curves on a triangulated surface $S$ are homotopic.
\end{prop}
\begin{proof}
The algorithm works as follows:
\begin{itemize}
\item[(I)] If necessary, isotope the curves so that all the intersections are transverse.
\item[(II)] While there exists a pair of intersection points which bounds a disk {\bf do}:
\begin{itemize}
\item[ ] push one of the curves through the disk to remove the two intersection points. {\bf end do}
\end{itemize}
\item[(III)] Does there exist an intersection point? If yes, output {\bf the curves are not homotopic} and halt. If no,
proceed to step (IV).
\item[(IV)] Do the two curves bound an annulus? If no, output {\bf the curves are not homotopic} and halt. If yes,
output {\bf the curves are homotopic} and halt.
\end{itemize}
To verify the algorithm, we note that the Bigon Criterion implies that step (II) can be performed until the curves are in a minimal position.
The correctness of the algorithm now follows by \lemref{homotopy check1}.

\end{proof}

\paragraph{\bf Maps isotopic to the identity.}

The following theorem of Ladegaillerie \cite{Lad} will be useful
to us in what follows:

\begin{thm}
\label{thm-isotopy-check}
Let $K$ be a compact topological 1-complex, X an oriented compact surface with boundary, $i_{0}, i_{1}$ two embeddings of K into the interior of X. There is an equivalence between the two following properties:
\begin{enumerate}
\item $i_{0}$ and $i_{1}$ are isotopic by an ambient isotopy of $X$ (fixed on $\partial X$)
\item $i_{0}$ and $i_{1}$ are homotopic and there is an orientation preserving homeomorphism $h:X\to X$ such that $h \circ i_{0}=i_{1}$.
\end{enumerate}
\end{thm}
 
We formulate the following corollary:
\begin{prop}
\label{identify-isometry}
There exists an algorithm $A$ which does the following. Given a triangulated sphere with a finite number of punctures $S=S^2-Z$ and a triangulated homeomorphism
$h:S\to S$, the algorithm identifies whether $h$ is isotopic to the identity.
\end{prop}
\begin{proof}
Let $x\notin Z$ be a vertex in the triangulation $\cT$. Consider a collection of closed loops $\gamma_i$ in $\partial \cT$ passing through the basepoint $x$
such that $\{\gamma_i\}$ forms a basis of $\pi_1(S)$ (refine the triangulation, if necessary). By \thmref{thm-isotopy-check}, it is sufficient to verify that
$h(\gamma_i)$ is homotopic to $\gamma_i$ for all $i$. Indeed, this is equivalent to the existence of a global isotopy of $S$ which moves $h(\gamma_i)$ to $\gamma_i$.
By the Alexander's trick, the latter statement means that $h$ is isotopic to the identity.
 
\end{proof}

\paragraph{\bf Dehn twists.} Recall the definition of a Dehn twist. Let $\gamma$ be a simple closed curve on a surface $S$, and let $A$ be a tubular neighborhood
of $\gamma$. Choose a homeomorphism 
$$h:S^1\times [0,1]\to A,$$
which endows the annulus $A$ with a coordinate system $(\theta,r)$ where $\theta\in\RR/\mod 2\pi\ZZ$ is the angular coordinate in $S^1$, and $r\in[0,1]$.
A {\it Dehn twist about $\gamma$} is the homeomorphism 
$$f:S\to S$$
which is identical outside $A$, and is given by 
$$f:h(\theta,r)\mapsto h(\theta+2\pi r, r).$$

\begin{figure}
\centerline{\includegraphics[width=0.8\textwidth]{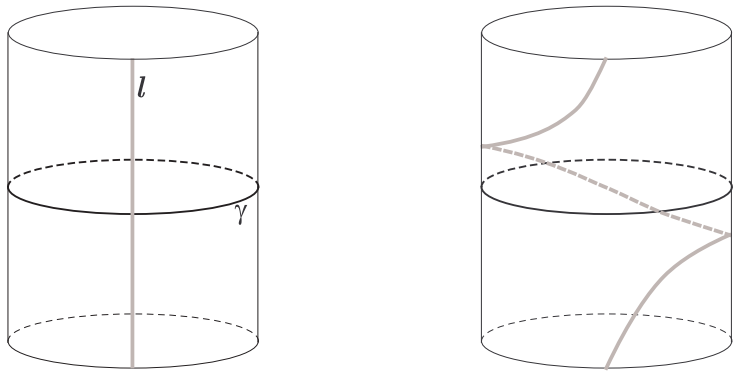}\label{fig-dehn-twist}}
\caption{Dehn twist of a cylinder about an equatorial curve $\gamma$: a vertical line $l$ before (left) and after (right) the twist is shown.}
\end{figure}

\paragraph{\bf Mapping class group.}
Since Thurston equivalence involves isotopies preserving pointwise the points of $P_{f}$,we are led to consider the \textit{pure mapping class group} $\text{PMod}(S^{2}-P_{f})$. 
It is the group of homeomorphisms of $S^{2}-P_{f}$ fixing $P_{f}$ pointwise,  modulo isotopies fixing $P_{f}$ pointwise.
The mapping class group acts on the set of isotopy classes of simple closed curves. 

We use the following fact:
\begin{prop}
\label{pmod}
The group $\text{PMod}(S^{2}-P_{f})$ is generated by a finite number of explicit Dehn twists.

\end{prop}

The finiteness of the number of generating twists is a classical result of Dehn; Lickorish \cite{Lic} has made the construction explicit. See, for example, \cite{FM} for an exposition.

\section{Algorithm $A_1$: detecting an obstruction}
\label{section2}

\paragraph{\bf Enumeration of the multicurves.}

We first prove the following proposition:
\begin{prop}
\label{enumerate}
There exists an algorithm $A$ which enumerates all non-peripheral multicurves on $S^{2}-P_{f}$. 
\end{prop}

For ease of reference let us state the following elementary fact:
\begin{prop}
\label{fact1}
Let $S^*$ denote $S^2$ with a finite number of punctures. Consider two simple closed curves $\gamma_1$ and $\gamma_2$ in $S^*$.
Assume that a component of $S^*-\gamma_1$ contains the same number of punctures as some component of $S^*-\gamma_2$. Then there exists a self-homeomorphism of $S^*$
which sends $\gamma_1$ to $\gamma_2$.
\end{prop}

We fix a finite collection of simple closed curves $c_1,\ldots, c_M$ so that the Dehn twists $T_{1}, \ldots, T_{M}$ around those curves generate the mapping class group 
$\text{PMod}(S^{2}-P_{f})$.  This construction can be performed algorithmically by \propref{pmod}.
We further refine our initial triangulation of the sphere so that these Dehn twists can be considered as piecewise linear maps relatively to the refined triangulation.

For every set of $j$ punctures with  $j \in \{2, \ldots , \vert P_{f}\vert-2\}$ we choose one polygonal
simple closed curve  which separates them from the rest of $P_f$. Denote these curves $\gamma_k$ (a simple count shows that $k=1\ldots 2^{|P_f|-1}-|P_f|-1$.

\begin{proof}[Proof of \propref{enumerate}]
To enumerate all multicurves, we proceed inductively as follows.
At step $0$, our collection of multicurves consists of all finite subsets of the set $\{\gamma_k\}$.

At step $N \in \NN$, we generate all the possible images of the curves $\gamma_{k}$ by reduced words in the $T_{i}, T_{i}^{-1}$ of length less than $N$. 
Using \propref{homotopy check}, we remove all duplications from this finite collection. All inessential or peripheral curves are likewise removed.

 We then consider all finite subsets of this collection. Using
\propref{homotopy check} again, we remove all subsets which have been generated previously, at steps $0,1,\ldots,N-1$.

By \propref{fact1} and \propref{pmod}, every multicurve is thus generated.
\end{proof}

\paragraph{\bf Construction of the algorithm $A_1$.}
Denote $A(n)$ the algorithm of \propref{enumerate}, which generates the exhaustive sequence of multicurves $\Gamma_n$.
Set $n=1$.
\begin{itemize}
\item[(I)] use \propref{homotopy check} to check whether $\Gamma_n$ is invariant. If not, proceed to step (V).
\item[(II)] Compute the transition matrix $M_{\gamma\delta}$ of the associated Thurston linear transformation $f_{\Gamma_n}$. Denote
$P_n(\lambda)$ the characteristic polynomial of $M_{\gamma\delta}$.
\item[(III)] Is $1$ a root of $P_n$? If yes, go to step (VI).
\item[(IV)] For $1\leq j\leq n$ {\bf do}
\begin{itemize}
\item Use  $A_3$ to query whether $P_j$ has a root $\lambda$ with $B(\lambda,2^{-3i})\subset [1+2^{-i},\infty).$ If yes, go to step (VI).
\end{itemize} 
{\bf end  do}
\item [(V)] $n\mapsto n+1$. Return to step (I).
\item[(VI)] Return {\bf there exists a Thurston obstruction} and halt.
\end{itemize}

\section{Algorithm $A_2$: finding an equivalent rational map}
\label{section3}
\subsection{Moduli space of rational maps}
Let $\text{Rat}_{d}$ denote the space of all holomorphic maps of degree $d\geq 2$ to itself. 
These maps can be written as fractions $\frac{p(z)}{q(z)}$, where the polynomials $p,\; q$ 
are relatively prime and $d= \max(\textrm{deg }  p,\textrm{ deg } q)$. It can be shown that
  $\text{Rat}_{d}$ is a connected complex-analytic manifold of dimension $2d+1$. Denoting 
$\text{Res}(P,Q)$ the resultant of $P$ and $Q$,
 one can represent $\text{Rat}_{d}$ as the open set $\mathbb{P}^{2d+1}/ V$, 
where $V=\{ P,Q: \text{Res}(P,Q)=0\}$. (See for example \cite{Sil}, page 169).

Since we are interested in equivalence classes of rational maps under conjugation by 
M{\"o}bius maps, we are led to consider the moduli space 
$\mathcal{M}_{d}=\text{Rat}_{d}/ \text{PSL}_{2}(\mathbb{C})$ 
(observe that it is the same as $\text{Rat}_{d}/ \text{PGL}_{2}(\mathbb{C})$). 

We note the following easy fact:
\begin{prop}
The moduli space $\cM_d$ has the structure of a complex orbifold of dimension $2d-2$.
\end{prop}
\begin{proof}
The stabilizer of a rational map $f\neq \text{Id}$ under the action of $\text{PSL}_{2}(\mathbb{C})$ is the subgroup
$\cS(f)$ consisting of M{\"o}bius maps which commute with $f$. There exists $n\in\NN$ such that the set $P_n$
consisting of periodic points of $f$ with periods less than $n$ has at least three points. Since every $M\in\cS(f)$
must permute the points in $P_n$, the stabilizer $\cS(f)$ is necessarily finite.
\end{proof}
As an example of a rational map with a non-trivial stabilizer, consider $f(z)=z^d$ for $d>2$, in which case,
$$\cS(f)=\{z\mapsto \lambda z|\text{ where }\lambda^{d-1}=1\}.$$ 

We will now require a more computation-friendly description of $\cM_d$. There are several similar approaches to this in the
existing literature; we use the work \cite{PilC}. As a first step we note the
following standard fact:
\begin{prop}
Suppose $R$ is not conjugate to a map of the form $z\mapsto z^{\pm d}$. Then the union of the critical and the postcritical sets $C_f\cup P_f$ contains at least three points.
\end{prop}

\paragraph{\bf Mapping scheme.}

A {\it mapping scheme} of degree d is a triple $(N, \tau, \omega)$, where $N\in\NN$ and $N\geq 3$;
$\tau$ is a {\it dynamics function}
$$\{1,\ldots,N\}\overset{\tau}{\longrightarrow}\{1,\ldots,N\},$$
and  $\omega$ is a {\it local degree function}
$$\{1,\ldots,N\}\overset{\omega}{\longrightarrow}\NN.$$

For a postcritically finite branched covering map $f:S^2\to S^2$ denote $Z_{f}$ is the union $C_{f}\cup P_{f}$ of the critical set and the postcritical set of $f$.
We say that  $f$ {\it realizes} $X=(N,\tau,\omega)$ 
if we can choose a bijection $\psi:\{1,\ldots,N\}\to Z_f$ such that:
\begin{itemize}
\item  $\psi(\tau(x))=f(\psi(x))$;
\item  the local degree of $f$ at $\psi(x)$ is equal to $\omega(x)$.
\end{itemize}
Following \cite{PilC}, a {\it normalization}
of a  mapping scheme $X=(N,\tau, \omega)$ is an injection 
$$\alpha:\{0,1,\infty\}\to \{1,\ldots,N\}.$$
We will denote a pair $(X,\alpha)$ by $X_\alpha$, and refer to it as a {\it marked mapping scheme}.

We say that a rational map $f:\hat \CC\to \hat \CC$ realizes a marked mapping scheme $X_\alpha$ if the bijection $\psi$ as defined above
has the additional property:
\begin{itemize}
\item $\{0,1,\infty\}\subset Z_f\text{ and }\psi^{-1}(i)=\alpha(i).$
\end{itemize}
The set of all rational realizations of a specific marked mapping 
scheme $X^{*}=X_\alpha$ will be written as $\text{Rat}^{\times}(X^{*})$.

Viewing the points in $\{1,\ldots,N\}$ as vertices of a weighted directed graph with
arrows connecting $x$ with $\tau(x)$ having weights $\omega(x)$,
 we define a {\it signature} $\cN_X$ as follows.
The signature is the set
$$\cN_X=\{N(x),\; x\in\tau(\{1,\ldots,N\})\},$$
where $N(x)$ is defined as the least common multiple, over all directed paths
of any length joining $y$ to $x$, of the product of the weights of edges along this
path. The significance of this definition lies in the following:
\begin{prop}[\cite{DH}]
If $f$ is a rational mapping which realizes the mapping scheme $X$, then $\cN_X$ is the signature of the  orbifold of $f$.
\end{prop}

We can now state:

\begin{thm}[\cite{PilC}]
\label{thm-scheme}
Given a mapping scheme $X$ whose signature  is not $(2,2,2,2)$ and a normalized marking scheme $X^{*}=(X,\tau,\omega)$,
  there is an injection $\iota :\text{Rat}^{\times}(X^{*})\rightarrow \mathbb{C}^{N}$ 
such that the image is a zero-dimensional affine variety 
$V(\mathcal{I})$ determined by an ideal $\mathcal{I}=\cI_{X^*}$, where $\mathcal{I}$ is defined over the rationals. In particular, 
$\text{Rat}^{\times}(X^{*})$ is finite. Furthermore, a basis for $\mathcal{I}$ can be algorithmically computed.
\end{thm}

Let us give an indication of how the injection $\iota$ may be defined.
Enumerate the elements of $Z_{f}$ as $\{ z_{i}\}$ so that $\psi(i)=z_{i}$.
Isolate the elements sent respectively to zero and the infinity as follows: 
$$FZ=\{ n \vert z_{n}\in Z_{f}-\{\infty\} \textrm{ and } f(z_{n})=0\},$$ 
$$\text{ and }FP=\{ m \vert z_{m}\in Z_{f}-\{\infty\} \textrm{ and} f(z_{m})=\infty\}.$$ 
Then any normalized rational realization $f$ can be written {\sl uniquely} as
\begin{equation}
\label{normalized map}
f(z)=\frac{a_{0}+\ldots +  a_{r}z^{r}}{b_{0}+ b_{1}z+ \ldots + 1\cdot z^{s}}\cdot \frac{\prod_{n \in FZ}(z-z_{n})^{d_{n}}}{\prod_{m \in FP}(z-z_{m})^{d_{m}}}
\end{equation}
Thus, the coordinates $a_j$, $b_j$, and $z_j$, specify the normalized rational map as a point in $\CC^N$.

\subsection{The algorithm $A_2$}
We start with a piecewise-linear Thurston map $f:S^2\to S^2$ with triangulation $\cT$. 
Let $X$ be a mapping scheme which is realized 
by $f$, and let $\alpha$ be a marking of $X$. Set $X^*=X_\alpha$.
Let $N$ and  $\cI_{X^*}$ be as in \thmref{thm-scheme}. Compute the finite set $g=(g_{1}, \ldots g_{r})$ of polynomials with rational coefficients, which generates the ideal $I_{X^*}$. Denote $A$ the algorithm of \propref{enumerate}. Let $\{R_1,\ldots,R_m\}$ be the finite collection of rational maps which realize the marked
mapping scheme $X^*$.

The algorithm $A_2$ works as follows:
\begin{itemize}
\item[(I)] use $A$ to enumerate as $D_i^k$ the representatives of the mapping class groups $\text{PMod}(S^2-P_{R^k})$;
\item[(II)] for every $1\leq k\leq m$ and every pair $(i,j)\subset \NN\times \NN$ {\bf do}
\begin{itemize}
\item discretize the map $(D_j)^{-1}\circ R_k\circ D_j$ to a piecewise linear Thurston mapping $M_j^k$ with
triangulation $\cT_1$. Note that at this stage the algorithm $A_3$ may need to be invoked to better estimate
the coefficients of $R_k$;
\item refine the triangulations $\cT$ and $\cT_1$ to obtain a triangulation $\cT_2$ on which both $f$ and $M_{i,j}^k$ are defined;
\item identify all triangulated orientation preserving 
homeomorphisms $h_j$ of $\cT_2$; use the algorithm of \propref{identify-isometry} to list all $w_k=h_{j_k}$ which are
isotopic to the identity;
\item perform a finite check to determine whether there exists a pair $w_i$, $w_j$ such that 
$$w_i\circ M_{i,j}^k=w_j\circ f.$$ 
If yes, go to (III). {\bf End do} 
\end{itemize} 
\item[(III)] output {\bf $f$ is Thurston equivalent to a rational map}, output the isolating neighborhood for $R_k$, and exit the algorithm.
\end{itemize}

\section{Concluding remarks}

Let us note an easy corollary of our main result. Consider the following decidability problem:

\medskip
\noindent
{\bf Problem (a):} {\sl Given two piecewise linear Thurston mappings $f$ and $g$ with hyperbolic orbifolds and
without Thurston obstructions, are $f$ and $g$ Thurston equivalent?}

\medskip
\noindent

\begin{thm}
Problem (a) is algorithmically decidable.

\end{thm}

\begin{proof}
Denote $R_f$ and $R_g$ the rational maps equivalent to $f$ and $g$ respectively. The existence of such maps
is guaranteed by Thurston's theorem. They are defined up to a M{\"o}bius conjugacy. 
The algorithm works as follows:
\begin{itemize}
\item[(I)] Check if $f$ and $g$ have identical mapping schemes $X=(N,\tau,\omega)$. If not, output {\bf the maps are not Thurston equivalent} and exit.
\item[(II)] Run the algorithm $A_2$ to find isolating $2^{-n}$-neighborhoods $U_f$ and $U_g$ of $R_f$ and $R_g$ respectively 
in the parameter space of normal forms (\ref{normalized map}).
\item[(III)] For each ordered triple of distinct natural numbers $(a,b,c)$ between $1$ and $N$ {\bf do}
\begin{itemize}
\item Normalize the mapping scheme $X$ by $\alpha:(0,1,\infty)\mapsto (a,b,c)$. Calculate a point of the parameter space $\bar w_{(a,b,c)}$ 
representing the normal form (\ref{normalized map}) of $R_f$ corresponding to $(X,\alpha)$ with precision $2^{-(n+2)}$. 
\item If $\bar w_{(a,b,c)}\in U_g$ then output {\bf the maps are Thurston equivalent} and exit.
\end{itemize}

\item[(IV)] Output {\bf the maps are not Thurston equivalent} and exit.
\end{itemize}
\end{proof}

Consider the following natural generalization of Problem (a), suggested to us by M. Lyubich:

\medskip
\noindent
{\bf Problem (b):} {\sl Given two piecewise linear Thurston mappings $f$ and $g$ with hyperbolic orbifolds, are $f$ and $g$ Thurston equivalent?}

\medskip
\noindent
Algorithmic decidability of Problem (b) presents an interesting direction of further study. 

We find that recognizability problems of combinatorial equivalence of piecewise-linear maps are analogous
to recognizability problems of piecewise-linear manifolds,
which are solvable in dimensions $3$ or less, and become algorithmically intractable in dimensions greater than $4$
(and possibly $4$ as well) -- see, for example, the book of Weinberger \cite{Wein}.
In conclusion, we speculate that natural notions of equivalence of maps in higher dimensions will lead to algorithmically
unsolvable problems -- and a new interplay between Dynamics and Computability.

\appendix
\section{A root-finding algorithm.} 
\label{section-root-finding}
The existence of a root-finding algorithm $A_3$ is a classical result of H. Weyl \cite{Wey}. 
Consider a system of analytic functions $g=\{g_i(z),\; i=1,\ldots,n\}$ defined in a box $D\subset\CC^n$.
  Denote $N(d,D)$ the number of common zeroes of $g$ in $D$, and assume that there are no common zeroes 
  on the boundary.
Then $N(d,D)$ can be computed using the following  multidimensional residue formula (see \cite{Sha} page 324):
\[
N(g,D)=\frac{(n-1)!}{(2\pi i)^{n}}\int_{\partial D}\frac{1}{\vert g \vert^{2n}}.\sum_{j=1}^{n}\overline{g_{j}}dg_{j}\wedge d\overline{g_{j}}\wedge dg_{1}\wedge d\overline{g_{1}} \ldots \overbrace{[j]}\ldots d\overline{g_{r}}\wedge dg_{r},
\] 
where $\overbrace{[j]}$ means that we omit the term $dg_{j}\wedge d\overline{g_{j}}$.

Observe that our space $Rat^{\times}(X^{\star})$ is indeed given as the zeroes of $N$ polynomials in $\CC^{N}$.

Weyl's algorithm to locate all roots of $g=0$ in an isolating neighborhood $G$ works as follows.
Begin by setting $j=0$ covering $G$ by a cubic grid of size $2^{-j}=1$. In each cube $C_k$ of the grid, 
use the residue formula
to check whether there are any zeros in it. 
Since we cannot catch zeros on the boundary of a cube, perform the check for a cube of twice the size --
it is guaranteed to catch any zeros in the closure of $C_k$. 

Throw away all cubes without any zeros. Increment $j\mapsto j+1$ and divide the remaining cubes $C_k$
 into cubes with side $2^{-j}$. Repeat the process, until all zeros are identified with the desired precision.

\newpage


\begin{bibdiv}
\begin{biblist}


\bib{Bin}{article}{
author={R. H. Bing},
title={An alternative proof that 3-manifolds can be triangulated},
journal={Annals of Math. (2)},
volume={69},
pages={37--65},
date={1959},
}



\bib{PilC}{article}{

title={A census of rational maps},
author={E. Brezin},
author={R. Byrne},
author={J. Levy},
author={K. Pilgrim}, 
author={K. Plummer},
journal={Conformal Geometry and Dynamics},
volume={4},
date={2000},
pages={35--74}
}

\bib{DH}{article}{
title={A proof of Thurston's topological characterization of rational functions},
author={A. Douady and J.H. Hubbard},
journal={Acta Math.},
volume={171},
date={1993},
pages={263--297}
}
\bib{Doug}{book}{
title={Alg\`ebre et th\'eories galoisiennes},
author={R. Douady and A. Douady},
publisher={Cassini},
date={2005},
}
\bib{FM}{book}{
title={A Primer on mapping class groups},
author={B. Farb and D. Margalit},
}
\bib{HS}{article}{
author={J.H. Hubbard and D. Schleicher},
title={The Spider Algorithm}
}
\bib{Kam}{article}{
author={A. Kameyama}, 
title={The Thurston equivalence for postcritically finite branched coverings},
journal={Osaka J. Math.},
volume={38},
date={2001},
pages={565--610}

}

\bib{Lad}
{article}{
author={Y. Ladegaillerie},
title={Classes d'isotopie de plongements de 1-complexes dans les surfaces},
journal={Topology},
volume={23},
date={1984},
pages={303--311}
}

\bib{Lic}
{article}{
author={W. B. R. Lickorish},
title={A finite set of generators for the homeotopy group of a 2-manifold}, 
journal={Proc. Cambridge Philos. Soc.},
volume={60},
date={1964}, 
pages={769--778}
}


\bib{MT}{article}{
author={J. Milnor and W. Thurston},
title={On iterated maps of the interval}
}

\bib{Nek}{book}{
title={Self-similar groups},
author={V. Nekrashevych},
publisher={American Mathematical Society},
volume={117},
date={2005},
}
\bib{Pil}{book}{
title={Combinations of complex dynamical systems},
author={K. Pilgrim},
publisher={Springer},
series={Lecture Notes in Mathematics},
volume={1827},
date={2003},
}

\bib{Pil2}{article}{
author={K. Pilgrim}, 
title={An algebraic formulation of Thurston's combinatorial equivalence},
journal={Proc. Amer. Math. Soc.},
volume={131},
date={2003},
pages={3527--3534}
}
\bib{Rad}{article}{
author={T. Rad\'o},
title={Uber den Begriff der Riemannschen Fl\"achen.},
journal={Acta Litt. Sci. Szeged},
pages={101--121},
date={1925},
}

\bib{Sha}{book}{
title={Introduction \`a l'analyse complexe},
author={B. Shabat},
publisher={Mir},
date={1990},
}

\bib{Sil}{book}{
title={The arithmetic of dynamical systems},
author={J. H. Silverman},
publisher={Springer},
volume={241},
date={2007},
}

\bib{thu}{book}{
author={W. P. Thurston},
title={Three-dimensional Geometry and Topology},
publisher={Princeton University Press},
date={1997},
}
\bib{Wein}{book}{
author={S. Weinberger}, 
title={Computers, rigidity, and moduli. The large-scale fractal geometry of Riemannian moduli space. M. B. Porter Lectures},
publisher={Princeton University Press, Princeton, NJ},
date={2005},
}
\bib{Wey}{article}{
author={H. Weyl},
title={Randbemerkungen zu Hauptproblemen der Mathematik, II, Fundamentalsatz der Algebra
and Grundlagen der Mathematik},
journal={Math. Z.}, 
volume={20},
date={1924},
pages={131--151},
}

\end{biblist}
\end{bibdiv}
\end{document}